\documentclass[10pt]{amsart}
\usepackage{amsmath,amssymb, amsthm} 

\theoremstyle{plain}
    \newtheorem{thm}{Theorem}
    \newtheorem{cor}[thm]{Corollary}
    \newtheorem{conjecture}{Conjecture}
    \newtheorem{prop}{Proposition}
    
    \newtheorem{lemma}[prop]{Lemma}

    \newtheorem*{examplethm}{Example Theorem}
\theoremstyle{definition}
    \newtheorem{defn}[prop]{Definition}
\theoremstyle{remark}
    \newtheorem{rem}[prop]{Remark}

\def\tr{{\mathrm{tr}}}

\def\UH{\mathcal{UH}}

\def\X{\mathbb{X}}

\def\SO{{\mathrm{SO}}}

\def\be{\begin{equation}}
\def\ee{\end{equation}}

\def\bm{\begin{matrix}}
\def\em{\end{matrix}}

\newcommand{\C}{\mathbb{C}}\newcommand{\R}{\mathbb{R}}\newcommand{\Q}{\mathbb{Q}}\newcommand{\Z}{\mathbb{Z}}\newcommand{\N}{\mathbb{N}}

\newcommand{\D}{\mathbb{D}}
\renewcommand{\P}{\mathbb{P}}

\renewcommand{\setminus}{\smallsetminus}
\renewcommand{\emptyset}{\varnothing}
\renewcommand{\Im}{\mathrm{Im}\;}
\newcommand{\SL}{\mathrm{SL}}

\newcommand{\id}{\mathit{id}}

\def\sl{\mathrm{{sl}}}

\newcommand{\comm}[1]{}

\begin{document}

\title[Stratified analyticity of the Lyapunov exponent]
{Global theory of one-frequency Schr\"odinger operators I:
stratified analyticity of the Lyapunov exponent and the boundary of
nonuniform hyperbolicity}

\author[A.~Avila]{Artur Avila}
\address{CNRS UMR 7599,
Laboratoire de Probabilit\'es et Mod\`eles al\'eatoires.
Universit\'e Pierre et Marie Curie--Bo\^\i te courrier 188.
75252--Paris Cedex 05, France}
\thanks{This work was partially conducted during the period the author
served as a Clay Research Fellow.}
\curraddr{IMPA,
Estrada Dona Castorina 110, Rio de Janeiro, Brasil} 
\urladdr{www.impa.br/$\sim$avila/}
\email{artur@math.sunysb.edu}

\begin{abstract}
We study Schr\"odinger operators with a one-frequency analytic
potential, focusing on the transition
between the two distinct local regimes
characteristic respectively of large and small potentials.  From the
dynamical point of view, the transition signals the emergence of nonuniform
hyperbolicity, so the dependence of the Lyapunov exponent with respect to
parameters plays a central role in the analysis.
Though often ill-behaved by conventional measures, 
we show that the Lyapunov exponent is in fact remarkably regular in a
``stratified sense'' which we define: the irregularity comes from
the matching of nice (analytic or smooth) functions along sets with
complicated geometry.
This result allows us to stablish
that the ``critical set'' for the transition
has at most codimension one, so for a typical potential
the set of critical energies
is at most countable, hence typically
not seen by spectral measures.
Key to our approach are two results about the dependence
of the Lyapunov exponent of one-frequency $\SL(2,\C)$ cocycles
with respect to perturbations in the imaginary direction:
on one hand there is a severe
``quantization'' restriction, and on the other hand ``regularity'' of the
dependence characterizes uniform hyperbolicity when
the Lyapunov exponent is positive.  Our method is
independent of arithmetic conditions on the frequency.
\end{abstract}

\date{\today}

\maketitle




\section{Introduction}

This work is concerned with the dynamics of one-frequency
$\SL(2)$ cocycles, and has two distinct aspects: the analysis, from a new
point of view, of the dependence of the Lyapunov exponent with respect to
parameters, and the study of the ``boundary'' of nonuniform hyperbolicity.
But our underlying motivation is to build a global theory of one-frequency
Schr\"odinger operators with general analytic potentials, so we will start
from there.

\subsection{One-frequency Schr\"odinger operators}

A one-dimensional quasiperiodic
Schr\"odinger operator with one-frequency analytic potential
$H=H_{\alpha,v}:\ell^2(\Z) \to \ell^2(\Z)$ is given by
\be
(Hu)_n=u_{n+1}+u_{n-1}+v(n \alpha) u_n,
\ee
where $v:\R/\Z \to \R$ is an analytic function (the potential), and $\alpha
\in \R \setminus \Q$ is the frequency.
We denote by $\Sigma=\Sigma_{\alpha,v}$ be the spectrum of $H$.
Despite many recent advances (\cite {BG}, \cite {GS}, \cite {B},
\cite {BJ1}, \cite
{BJ2}, \cite {AK1}, \cite {GS1},
\cite {GS2}, \cite {AJ2}, \cite {AFK}, \cite {A})
key aspects of an authentic
``global theory'' of such operators have been missing.
Namely, progress
has been made mainly into the understanding of the behavior in regions
of the spectrum belonging to two regimes with (at least some of the)
behavior caracteristic,
respectively, of ``large'' and ``small'' potentials.
But the transition between the two regimes has
been considerably harder to understand.

Until now, there has been only one case where the analysis has genuinely
been carried out at a global level.
The almost Mathieu operator, $v(x)=2
\lambda \cos 2 \pi (\theta+x)$, is a highly symmetric model for which
coupling strengths $\lambda$ and $\lambda^{-1}$ can be related
through the Fourier transform (Aubry duality).  Due to this unique feature,
it has been possible to stablish that the transition happens precisely at
the (self-dual) critical coupling $|\lambda|=1$: in the subcritical regime
$|\lambda|<1$ all energies in the
spectrum behave as for small potentials, while in the supercritical regime
$|\lambda|>1$ all energies in the spectrum behave as for large potentials.
Hence typical almost Mathieu operators fall entirely in one
regime or the other.  Related to this simple phase
transition picture, is the fundamental spectral result of \cite {J}, which
implies that the spectral measure of a typical Almost Mathieu operator
has no singular continuous components (it is either typically
atomic for $|\lambda|>1$ or typically absolutely continuous for
$|\lambda|<1$).

One precise way to distinguish the subcritical and the supercritical regime
for the almost Mathieu operator is by means of the Lyapunov exponent. 
Recall that for $E \in \R$, a formal solution $u \in \C^\Z$
of $Hu=Eu$ can be reconstructed from
its values at two consecutive points by application of $n$-step transfer
matrices:
\be
A_n(k \alpha) \cdot \left (\bm u_k \\ u_{k-1} \em \right )=
\left (\bm u_{k+n} \\ u_{k+n-1} \em \right ),
\ee
where $A_n:\R/\Z \to \SL(2,\R)$, $n \in \Z$,
are analytic functions defined on the same band of analyticity of $v$, given
in terms of $A=\left (\bm E-v & -1 \\ 1 & 0 \em \right )$ by
\be \label {2}
A_n(\cdot)=A(\cdot+(n-1) \alpha) \cdots A(\cdot),\; A_{-n}(\cdot)=A_n(\cdot-n \alpha)^{-1},\;
n \geq 1,\; A_0(\cdot)=\id,
\ee
The Lyapunov exponent at energy $E$ is denoted by $L(E)$ and given by
\be \label {3}
\lim_{n \to \infty} \frac {1} {n} \int_{\R/\Z} \ln \|A_n(x)\| dx \geq 0.
\ee
It follows from the Aubry-Andr\'e formula (proved by Bourgain-Jitomirskaya
\cite {BJ1}) that $L(E)=\max \{0,\ln |\lambda|\}$ for $E \in
\Sigma_{\alpha,v}$.  Thus the supercritical regime can be
distinguished by the positivity of the Lyapunov exponent:
supercritical
just means {\it nonuniformly hyperbolic} in dynamical systems terminology.

How to distinguish subcritical energies from critical ones
(since both have zero Lyapunov exponent)?  One way could be in terms of
their stability: critical energies are in the boundary of the
supercritical regime, while subcritical ones are far away.
Another, more
intrinsic way, consists of looking at the complex extensions of the $A_n$:
it can be shown (by a combination of \cite {J} and \cite {JKS}) that
for subcritical energies we have a uniform subexponential bound
$\ln \|A_n(z)\|=o(n)$ through a band $|\Im z|<\delta(\lambda)$, while for
critical energies this is not the case (it follows from \cite {H}).
(See also the Appendix for a rederivation of both facts in the spirit of
this paper.)

This being said, this work is not concerned with the almost Mathieu,
whose global theory is very advanced.  Still, what we know about it
provides a powerful hint about how to approach the general theory.  By
analogy, we can always classify energies in the spectrum of an operator
$H_{\alpha,v}$ as supercritical, subcritical, or critical in terms of the
growth behavior of (complex extensions of) transfer matrices,\footnote{
That large
potentials fall into the supercritical regime then follows from \cite {SS}
and that small potentials fall into the subcritical one is a consequence of
\cite {BJ1} and \cite {BJ2}.} though differently from the almost Mathieu
case the coexistence of regimes is possible \cite {Bj1}.
Beyond the ``local'' problems of describing precisely the behavior at the
supercritical and subcritical regimes, a proper global theory should
certainly explain how the ``phase transition'' between them occurs,
and how this
critical set of energies affects the spectral analysis of $H$.

In this direction, our main result can be stated as follows.
Let $C^\omega_\delta(\R/\Z,\R)$ be the real Banach space of analytic
functions $\R/\Z \to \R$ admitting a holomorphic extension to $|\Im
z|<\delta$ which is continuous up to the boundary.

\begin{thm} \label {cod}

For any $\alpha \in \R \setminus \Q$, the set of potentials and energies
$(v,E)$ such that $E$ is a critical energy for $H_{\alpha,v}$ is contained in
a countable union of codimension-one analytic submanifolds of
$C^\omega_\delta(\R/\Z,\R) \times \R$.\footnote {A codimension-one analytic
submanifold is a (not-necessarily closed) set $X$ given locally (near any
point of $X$) as the zero set of an analytic submersion
$C^\omega_\delta(\R/\Z,\R) \to \R$.}

\end{thm}

In particular, a typical operator $H$ will have at most countably many
critical energies.  In the continuation of this series \cite {A3},
this will be the starting point of the proof that the critical
set is typically empty.

It was deliberately implied in the discussion above that the non-critical
regimes were stable (with respect to perturbations of
the energy, potential or
frequency), but the critical one was not.  Stability of nonuniform
hyperbolicity was known (continuity of the Lyapunov exponent \cite
{BJ1}), while the stability of the subcritical regime is obtained here.  The
instability of the critical regime of course follows from Theorem \ref
{cod}.  The stability of the subcritical regime
implies that the critical set contains the boundary of
nonuniform hyperbolicity.\footnote {It is actually true that any critical
energy can be made supercritical under an arbitrarily small perturbation of
the potential, see \cite {A3}.}

In the next section we will decribe our results about
the dependence of the Lyapunov exponent with respect to parameters which
play a key role in the proof of Theorem \ref {cod}
and have otherwise independent interest.  In Section \ref {almost
reducibility},
we will further comment on how our work on
criticality relates to the spectral analysis of the operators, and in
particular how it gives a framework to address the following generalization
of Jitomirskaya's work \cite {J} about the almost Mathieu operator.

\begin{conjecture} \label {sc}

For a (measure-theoretically) typical operator $H$, the spectral measures
have no singular continuous component.

\end{conjecture}

\subsection{Stratified analyticity of the Lyapunov exponent}

As discussed above, the Lyapunov exponent $L$ is
fundamental in the understanding of the spectral properties of $H$.  It is
also closely connected with another important quantity, the
{\it integrated density of states} (i.d.s.)
$N$.  As the Lyapunov exponent, the
(i.d.s.) is a function of the energy: while
the Lyapunov exponent measures the asymptotic
average growth/decay of solutions (not necessary in $\ell^2$) of the
equation $Hu=E u$, the integrated density of states gives the
asymptotic distribution of eigenvalues of restrictions to large boxes.  Both
are related by the Thouless formula:
\be
L(E)=\int \ln |E'-E| dN(E').
\ee

Much work has been dedicated to the regularity properties
of $L$ and $N$.  For quite general reasons,
the integrated density of states is a continuous increasing function onto
$[0,1]$, and it is constant outside the spectrum.  Notice that this is not
enough to conclude continuity of the Lyapunov exponent from the
Thouless formula.  Other regularity properties (such as H\"older), do pass
from $N$ to $L$ and vice-versa.  This being said, our focus here is
primarily on the Lyapunov exponent on its own.

It is easy to see that the Lyapunov exponent is real analytic outside the
spectrum.  Beyond that, however,
there are obvious limitations to the regularity of the
Lyapunov exponent.  For a constant potential, say $v=0$, the Lyapunov
exponent is $L(E)=\max \{0, \ln \frac {1} {2} (E+\sqrt {E^2-4})\}$, so it is
only $1/2$-H\"older continuous.  With Diophantine frequencies and small
potentials, the generic situation is to have
Cantor spectrum with countably many square
root singularities at the endpoints of gaps \cite {E}.
For small potential and generic frequencies, it is possible to show that the
Lyapunov exponent escapes any fixed continuity modulus (such as H\"older),
and it is also not of bounded variation.
More delicately, Bourgain \cite {B} has observed that in the case of the
critical almost Mathieu operator
the Lyapunov exponent needs not be H\"older even for Diophantine
frequencies (another instance of complications arising
at the boundary of non-uniform hyperbolicity).  Though a surprising result,
analytic regularity, was obtained in a related, but non-Schr\"odinger,
context \cite {AK2}, the negative results described above seemed to impose
serious limitations on the amount of regularity one should even try to look
for in the Schr\"odinger case.

As for positive results, a key development was the proof by Goldstein-Schlag
\cite {GS} that the Lyapunov exponent is H\"older continuous
for Diophantine frequencies in the regime of
positive Lyapunov exponent.  Later Bourgain-Jitomirskaya \cite {BJ1}
proved that the
Lyapunov exponent is continuous for all irrational frequencies, and this
result played a fundamental role in the recent theory of the almost Mathieu
operator.  More
delicate estimates on the H\"older regularity for Diophantine frequencies
remained an important topic \cite {GS1}, \cite {AJ2}.

There is however one important case where, in a different sense,
much stronger regularity holds.  For small analytic potentials, it follows
from the work of Bourgain-Jitomirskaya (\cite {BJ1} and \cite {BJ2}, see
\cite {AJ2}) that the Lyapunov exponent is zero
(and hence constant) in the spectrum!  In general, however, the Lyapunov
exponent need not be constant in the spectrum.  In fact, there are examples
where the Lyapunov exponent vanishes in part of the spectrum and is positive
in some other part \cite {Bj}.
Particularly in this positive Lyapunov exponent regime,
it would seem unreasonable,
given the negative results outlined above, to expect much more regularity. 
In fact, from a dynamical systems perspective, it would be natural to expect
bad behavior in such setting, since when the Lyapunov exponent is positive,
the associated dynamical system in the
two torus presents ``strange attractors'' with
very complicated dependence of the parameters \cite {Bj1}.

In this respect, the almost Mathieu operator would seem to behave quite
oddly.  As we have seen, by the Aubry formula
the Lyapunov exponent is always
constant in the spectrum, and moreover, this constant is just a simple
expression of the coupling $\max \{0,\ln \lambda\}$ (and in particular, it
does become positive in the {\it supercritical regime} $\lambda>1$).
It remains true that
the Lyapunov exponent displays wild oscillations
``just outside'' the spectrum, so this is not inconsistent
with the negative results discussed above.

However, for a long time, the general feeling
has been that this just reinforces the special status of the almost
Mathieu operator (which admits a remarkable symmetry, Aubry duality,
relating the supercritical and the subcritical regimes), and
such a phenomenon would seem to have little to do with the case of
general potentials.  This general feeling is wrong, as the following sample
result shows.

\begin{examplethm} \label {example}

Let $\lambda>1$ and let $w$ be any real analytic function.
For $\epsilon \in \R$, let $v(x)=2 \lambda \cos 2 \pi x+
\epsilon w(x)$.  Then for $\epsilon$ small enough, for every $\alpha \in \R
\setminus \Q$, the Lyapunov exponent restricted to the spectrum is a
positive real analytic function.

\end{examplethm}

Of course by a real analytic function on a set
we just mean the restriction of some real analytic function defined on an
open neighborhood.

For an arbitrary real analytic potential, the situation is just slightly
lengthier to describe.  Let $X$ be a topological space.  A
stratification of $X$ is a strictly decreasing finite or countable sequence
of closed sets $X=X_0 \supset X_1 \supset \cdots$ such that
$\cap X_i=\emptyset$.  We call $X_i \setminus X_{i+1}$ the $i$-th {\it
stratum} of the stratification.

Let now $X$ be a subset of a real analytic manifold,
and let $f:X \to \R$ be a continuous function.  We say that
$f$ is $C^r$-stratified
if there exists a stratification such that the restriction of $f$ to each
stratum is $C^r$.

\begin{thm}[Stratified analyticity in the energy] \label {e}

Let $\alpha \in \R \setminus \Q$ and $v$ be any real analytic
function.  Then the Lyapunov exponent is a $C^\omega$-stratified
function of the energy.

\end{thm}

As we will see, in this theorem
the stratification starts with $X_1=\Sigma_{\alpha,v}$, which is
compact, so the stratification is finite.

Nothing restricts us to look only at the energy as a parameter.  For
instance, in the case of the almost Mathieu operator, the Lyapunov exponent
(restricted to the spectrum) is real analytic also in the coupling constant,
except at $\lambda=1$.

\begin{thm}[Stratified analyticity in the potential] \label {v}

Let $\alpha \in \R \setminus \Q$, let $X$ be a real analytic manifold, and
let $v_\lambda$, $\lambda \in X$, be a real analytic family of real
analytic potentials.  Then the Lyapunov exponent is a $C^\omega$-stratified
function of both $\lambda$ and $E$.

\end{thm}

It is quite clear how this result opens the doors for the analysis of the
boundary of non-uniform hyperbolicity, since parameters
corresponding to the vanishing of the Lyapunov exponent are contained in the
set of solutions of equations (in infinitely many variables)
with analytic coefficients.  Of course, one
still has to analyze the nature of the equations one gets, guaranteeing the
non-vanishing of the coefficients.  Indeed, in the subcritical regime,
the coefficients do vanish.  We will work out suitable expressions for the
Lyapunov exponent restricted to strata which will allow us to show
non-vanishing outside the subcritical regime.

In the case of the almost Mathieu operator, there is no dependence of the
Lyapunov exponent on the frequency parameter.  In general,
Bourgain-Jitomirskaya proved the Lyapunov exponent is a continuous function
of $\alpha \in \R \setminus \Q$.
This is a very subtle result, as the continuity is not in general uniform in
$\alpha$.  We will show that the Lyapunov exponent is also
$C^\infty$-stratified as a function of $\alpha \in \R \setminus \Q$.

\begin{thm} \label {frequen}

Let $X$ be a real analytic manifold, and let $v_\lambda$, $\lambda \in X$,
be a real analytic family of real analytic potentials.  Then the Lyapunov
exponent is a $C^\infty$-stratified function of $(\alpha,\lambda,E) \in (\R
\setminus \Q) \times X \times \R$.

\end{thm}

With $v$ as in the Example Theorem, the Lyapunov exponent is actually
$C^\infty$ as a
function of $\alpha$ and $E$ in the spectrum.

\subsection{Lyapunov exponents of $\SL(2,\C)$ cocycles}

In the dynamical systems approach, which we follow here,
the understanding of the Schr\"odinger operator is obtained through the
detailed description of a certain family of dynamical systems.

A (one-frequency, analytic) quasiperiodic $\SL(2,\C)$ cocycle is a
pair $(\alpha,A)$, where $\alpha \in \R$ and
$A:\R/\Z \to \SL(2,\C)$ is
analytic, understood as defining a linear skew-product acting on $\R/\Z
\times \C^2$ by $(x,w) \mapsto (x+\alpha,A(x) \cdot w)$.  The iterates of
the cocycle have the form $(n\alpha,A_n)$ where $A_n$ is given by (\ref
{2}).
The Lyapunov exponent $L(\alpha,A)$ of the cocycle $(\alpha,A)$ is given by
the left hand side of (\ref {3}).
We say that $(\alpha,A)$ is {\it uniformly hyperbolic} if there exist
analytic functions $s:\R/\Z \to \P\C^2$, called the unstable and stable
directions, and $n \geq 1$ such that for every $x \in \R/\Z$, $A(x) \cdot
u(x)=u(x+\alpha)$ and $A(x) \cdot s(x)=s(x+\alpha)$, and for every unit
vector $w \in s(x)$ we have $\|A_n(x) \cdot w\|<1$ and $\|A_n(x) \cdot
w\|>1$ (clearly $u(x) \neq s(x)$ for every $x \in \R/\Z$).
The unstable and stable directions are uniquely caracterized by
those properties, and clearly $u(x) \neq s(x)$ for every $x \in \R/\Z$.
It is clear that if $(\alpha,A)$ is uniformly hyperbolic
then $L(\alpha,A)>0$.

If $L(\alpha,A)>0$ but $(\alpha,A)$ is not uniformly
hyperbolic, we will say that $(\alpha,A)$ is {\it nonuniformly hyperbolic}.

Uniform hyperbolicity is a stable property: the set $\UH \subset \R \times
C^\omega(\R/\Z,\SL(2,\C))$ of uniformly hyperbolic cocycles is open. 
Moreover, it implies good behavior of the Lyapunov exponent: the restriction
of $(\alpha,A) \mapsto L(\alpha,A)$ to $\UH$ is a $C^\infty$ function of
both variables,\footnote {Since $\UH$ is not a Banach manifold, it might
seem important to be precise about what notion of smoothness is used here. 
This issue can be avoided by enlarging the setting to include $C^\infty$
non-analytic cocycles (say by considering a Gevrey condition), so that we
end up with a Banach manifold.  The smoothness of the Lyapunov exponent in
this context is a consequence of normally hyperbolic theory \cite {HPS}.}
and it is a pluriharmonic function of the second variable.
\footnote {This means that, in addition to being continuous, given
any family $\lambda
\mapsto A^{(\lambda)} \in \UH$, $\lambda \in \D$,
which holomorphic (in the sense that it is continuous and for every
$x \in \R/\Z$ the map $\lambda
\mapsto A^{(\lambda)}(x)$ is holomorphic), the map $\lambda \mapsto
L(\alpha,A^{(\lambda)})$ is harmonic.}  In fact regularity properties of the
Lyapunov exponent are consequence of the
regularity of the unstable and stable directions, which depend smoothly on
both variables (by normally hyperbolic theory \cite {HPS})
and holomorphically on the second variable (by a simple normality argument).

On the other hand, a
variation \cite {JKS} of \cite {BJ1}
gives that $(\alpha,A) \mapsto L(\alpha,A)$ is continuous as a function on
$(\R \setminus \Q) \times C^\omega(\R/\Z,\SL(2,\C))$.  It is important to
notice (and in fact, fundamental in what follows)
that the Lyapunov exponent is
not continuous on $\R \times C^\omega(\R/\Z,\SL(2,\C))$.
In the remaining of this section,
we will restrict our attention (except otherwise noted) to cocycles with
irrational frequencies.

Most important examples
are Schr\"odinger cocycles $A^{(v)}$, determined by a real analytic function
$v$ by $A^{(v)}=\left (\bm v & -1 \\ 1 &
0 \em \right )$.  In this notation, the Lyapunov exponent at energy $E$ for the
operator $H_{\alpha,v}$ becomes $L(E)=L(\alpha,A^{(E-v)})$.
One of the most basic aspects of the connection between spectral and
dynamical properties is that
$E \notin \Sigma_{\alpha,v}$
if and only if $(\alpha,A^{(E-v)})$ is uniformly hyperbolic.
Thus the analyticity of $E \mapsto L(E)$ outside of the spectrum just
translates a general property of uniformly hyperbolic cocycles.

If $A \in C^\omega(\R/\Z,\SL(2,\C))$ admits a holomorphic extension to $|\Im
z|<\delta$, then for $|\epsilon|<\delta$
we can define $A_\epsilon \in C^\omega(\R/\Z,\SL(2,\C))$ by
$A_\epsilon(x)=A(x+i
\epsilon)$.  The Lyapunov exponent $L(\alpha,A_\epsilon)$ is easily seen to
be a convex function of $\epsilon$.  Thus we can define a function
\be
\omega(\alpha,A)=\lim_{\epsilon \to 0+} \frac {1} {2 \pi \epsilon}
(L(\alpha,A_\epsilon)-L(\alpha,A)),
\ee
called the {\it acceleration}.
It follows from convexity and continuity of the Lyapunov exponent
that the acceleration is an upper semi-continuous
function in $(\R \setminus \Q) \times C^\omega(\R/\Z,\SL(2,\C))$.

Our starting point is the following result.

\begin{thm}[Acceleration is quantizatized] \label {quantized}

The acceleration of $\SL(2,\C)$ cocycles with irrational frequency
is always an integer.

\end{thm}

\begin{rem}

It is easy to see that quantization does not extend to rational
frequencies, see Remark \ref {rational}.

\end{rem}

This result allows us to break parameter spaces
into suitable pieces restricted to which we can study the dependence
of the Lyapunov exponent.

Quantization implies that
$\epsilon \mapsto L(\alpha,A_\epsilon)$ is a piecewise
affine function of $\epsilon$.  Knowing this, it makes sense to introduce
the following:

\begin{defn}

We say that
$(\alpha,A) \in (\R \setminus \Q) \times C^\omega(\R/\Z,\SL(2,\C))$ is
{\it regular} if $L(\alpha,A_\epsilon)$ is affine for
$\epsilon$ in a neighborhood of $0$.

\end{defn}

\begin{rem} \label {omega real}

If $A$ takes values in $\SL(2,\R)$ then $\epsilon \mapsto
L(\alpha,A_\epsilon)$ is an even function.  By convexity,
$\omega(\alpha,A) \geq 0$, and if $\alpha \in \R \setminus \Q$ then
$(\alpha,A)$ is regular if and only if $\omega(\alpha,A)=0$.

\end{rem}

Clearly regularity is an
open condition in $(\R \setminus \Q) \times C^\omega(\R/\Z,\SL(2,\C))$.

It is natural to assume that regularity has important
consequences for the dynamics.  Indeed, we have been able to completely
characterize the dynamics
of regular cocycles with positive Lyapunov exponent, which is the other
cornerstone of this paper.

\begin{thm} \label {uniformly hyperbolic}

Assume that $L(\alpha,A)>0$.  Then $(\alpha,A)$ is regular if and only if
$(\alpha,A)$ is uniformly hyperbolic.

\end{thm}

One striking consequence is the following:

\begin{cor} \label {alter}

For any $(\alpha,A) \in (\R \setminus \Q) \times C^\omega(\R/\Z,\SL(2,\C))$,
there exists $\epsilon_0>0$ such that
\begin{enumerate}
\item $L(\alpha,A_\epsilon)=0$ (and $\omega(\alpha,A)=0$) for every
$0<\epsilon<\epsilon_0$, or
\item $(\alpha,A_\epsilon)$ is uniformly hyperbolic for every
$0<\epsilon<\epsilon_0$.
\end{enumerate}

\end{cor}

\begin{proof}

Since $\epsilon \to L(\alpha,A_\epsilon)$ is piecewise affine, it must be
affine on $(0,\epsilon_0)$ for $\epsilon_0>0$ sufficiently small, so that
$(\alpha,A_\epsilon)$ is regular for every $0<\epsilon<\epsilon_0$.

Since the Lyapunov exponent is non-negative, if
$L(\alpha,A_\epsilon)>0$ for some $0<\epsilon<\epsilon_0$, then
$L(\alpha,A_\epsilon)>0$ for every $0<\epsilon<\epsilon_0$.  The result
follows from the previous theorem.
\end{proof}

As for the case of regular cocycles with zero Lyapunov exponent, this is the
topic of the Almost Reducibility Conjecture, which we will discuss in
section \ref {almost reducibility}.  For now, we will focus on the deduction of
regularity properties of
the Lyapunov exponent from Theorems \ref {quantized} and \ref {uniformly
hyperbolic}.

\subsubsection{Stratified regularity: proof of Theorems \ref {e}, \ref {v}
and \ref {frequen}}

For $\delta>0$, denote by
$C^\omega_\delta(\R/\Z,\SL(2,\C)) \subset
C^\omega(\R/\Z,\SL(2,\C))$ be the set of all $A$ which admit a bounded
holomorphic extension to $|\Im z|<\delta$, continuous up to the boundary. 
It is naturally endowed with a complex Banach manifold structure.

For $j \neq 0$, let
$\Omega_{\delta,j} \subset \R \times C^\omega_\delta(\R/\Z,\SL(2,\C))$
be the set of all $(\alpha,A)$ such that there exists $0<\delta'<\delta$
such that $(\alpha,A_{\delta'}) \in \UH$ and $\omega(\alpha,A_{\delta'})=j$,
and let $L_{\delta,j}:\Omega_{\delta,j} \to \R$ be given by
$L_{\delta,j}(\alpha,A)=L(\alpha,A_{\delta'})-2 \pi j \delta'$.  Since if
$0<\delta'<\delta''<\delta$,
$\omega(\alpha,A_{\delta'})=\omega(\alpha,A_{\delta''})=j$ implies that
$L(\alpha,A_{\delta'})=L(\alpha,A_{\delta''})-2 \pi j (\delta''-\delta')$,
we see that $L_{\delta,j}$ is well defined.

\begin{prop} \label {pluri}

$\Omega_{\delta,j}$ is open and $(\alpha,A) \mapsto L_{\delta,j}(\alpha,A)$
is a $C^\infty$ function, pluriharmonic in the second variable.  Moreover,
if $(\alpha,A) \in (\R \setminus \Q) \times
C^\omega_\delta(\R/\Z,\SL(2,\C))$ has acceleration $j$, then $(\alpha,A) \in
\Omega_{\delta,j}$ and $L(\alpha,A)=L_{\delta,j}(\alpha,A)$.

\end{prop}

\begin{proof}

The first part follows from openness of $\UH$ and
the regularity of the Lyapunov exponent restricted to $\UH$.  For the second
part, we use Corollary \ref {alter}
and upper semicontinuity of the Lyapunov exponent
to conclude that $\omega(\alpha,A)=j$ implies that $(\alpha,A_{\delta'}) \in
\UH$ and has acceleration $j$ for every $\delta'$ sufficiently small, which
gives also $L(\alpha,A)=L(\alpha,A_{\delta'})-2 \pi j \delta'$.
\end{proof}

We can now give the proof of Theorems \ref {e}, \ref {v} and \ref {frequen}.
For definiteness, we will consider
Theorem \ref {v}, the argument is exactly the same for the other theorems.
Define a stratification of the
parameter space $\X=\R \times X$:
$\X_0=\X$, $\X_1 \subset \X_0$
is the set of $(E,\lambda)$ such that
$(\alpha,A^{(E-v_\lambda)})$ is not uniformly hyperbolic and for $j \geq 2$,
$\X_j \subset \X_1$ is the set of $(E,\lambda)$ such that
$\omega(\alpha,A^{(E-v_\lambda)}) \geq j-1$.

Since uniform hyperbolicity is open and the acceleration is upper
semicontinuous, each $\X_j$ is closed, so this is indeed a stratification. 
Since the $0$-th stratum $X_0 \setminus X_1$
corresponds to uniformly hyperbolic cocycles, the
Lyapunov exponent is analytic there.

By quantization, the $j$-th stratum,
$j \geq 2$, corresponds to cocycles which are not uniformly hyperbolic and
have acceleration $j-1$.  For each $(E,\lambda_0)$ in such a stratum,
choose $\delta>0$ such that $\lambda \mapsto
A^{(v_\lambda)}$ is an analytic function in a neighborhood of $\lambda_0$. 
The analyticity of the Lyapunov exponent restricted to the stratum is then a
consequence of Proposition \ref {pluri}.

As for a parameter $(E,\lambda)$ in the
first stratum $X_1 \setminus X_2$, quantization implies that
$(\alpha,A^{(E-v_\lambda)})$ has non-positive acceleration, so by
Remark \ref {omega real}
$(\alpha,A^{(E-v_\lambda)})$ must be regular with zero acceleration.
Since it is not uniformly hyperbolic,
Theorem \ref {uniformly hyperbolic}
implies that $L(\alpha,A^{(E-v_\lambda)})=0$.  Thus the
Lyapunov exponent is in fact identically $0$ in the first stratum.\qed

\subsection{Codimensionality of critical cocycles}

Non-regular cocycles split into two groups, the ones with positive
Lyapunov exponent (non-uniformly hyperbolic cocycles), and the ones
with zero Lyapunov exponent, which we call {\it critical cocycles}.\footnote
{As explained before, this terminology is consistent
with the almost Mathieu operator terminology: it turns out that if
$v(x)=2 \lambda \cos 2 \pi (\theta+x)$, $\lambda \in \R$, then
$(\alpha,A^{(E-v)})$ is critical if and only if
$\lambda=1$ and $E \in\Sigma_{\alpha,v}$.}

As discussed before, the first group has been extensively studied recently
(\cite {BG}, \cite {GS}, \cite {GS1}, \cite {GS2}).
But very little is known about the second one.

Though our methods do not provide new information on the dynamics of
critical cocycles, they are perfectly adapted to show that critical cocycles
are rare.  This is somewhat surprising, since in dynamical systems,
it is rarely the case that the success of parameter exclusion precedes a
detailed control of the dynamics!

Of course, for $\SL(2,\C)$ cocycles, our previous results already
show that critical cocycles are rare in certain one-parameter families,
since for every $(\alpha,A)$, for every $\delta \neq 0$ small,
$(\alpha,A_\delta)$ is regular, and hence not critical.  But for our
applications we are mostly concerned with $\SL(2,\R)$-valued
cocycles, and even
more specifically, with Schr\"odinger cocycles.

If $(\alpha,A) \in (\R \setminus \Q) \times
C^\omega_\delta(\R/\Z,\SL(2,\C))$ is critical with acceleration $j$, then
$(\alpha,A) \in \Omega_{\delta,j}$ and $L_{\delta,j}=0$.  Moreover,
if $A$ is $\SL(2,\R)$-valued, criticality implies that the acceleration is
positive (see Remark \ref {omega real}).
So the locus of critical $\SL(2,\R)$-valued cocycles is covered
by countably many analytic sets $L_{\delta,j}^{-1}(0)$.  Thus the main
remaining issue is to show
that the functions $L_{\delta,j}$ are non-degenerate.

\begin{thm} \label {cod1}

For every $\alpha \in \R \setminus \Q$, $\delta>0$ and $j>0$, if
$v_* \in C^\omega_\delta(\R/\Z,\R)$ and
$\omega(\alpha,A^{(v_*)})=j$ then $v \mapsto L_{\delta,j}(\alpha,A^{(v)})$
is a submersion in a neighborhood of $v_*$.

\end{thm}

This theorem immediately implies Theorem \ref {cod}.

We are also able to show non-triviality in the case of
non-Schr\"odinger cocycles, see Remark \ref {ak2}:
though the derivative of $L_{\delta,j}$ may vanish,
this forces the dynamics to be particularly nice, and
it can be shown that the second derivative is non-vanishing.

\subsection{Almost reducibility} \label {almost reducibility}

The results of this paper give further motivation to
the research on the set of
regular cocycles with zero Lyapunov exponent.  The central problem here is
addressing the following conjecture.

\begin{conjecture}[Almost Reducibility Conjecture]

Regularity with zero Lyapunov exponents implies almost reducibility.  More
precisely, assume that $L(\alpha,A_\epsilon)=0$ for $a<\epsilon<b$.  Then
for every $n$ there exists a holomorphic map
$B_n:\{a+1/n<|\Im z|<b-1/n\} \to \SL(2,\C)$ such that
$\|B_n(z+\alpha) A(z) B(z)^{-1}-\id\|<1/n$ for $a+1/n<\Im z<b-1/n$.
Moreover, if $a=-b$ and $A$ is real-symmetric then each $B_n$ can be chosen
to be real-symmetric.

\end{conjecture}

This is a slightly more precise and general version than a conjecture first
made in \cite {AJ2}.
What makes this conjecture so central is that, in the real-symmetric
case, which is most important for our considerations,
almost reducibility was analyzed in much detail in recent works, see
\cite {AJ2}, \cite {A1}, \cite {AFK} and \cite {A}, so a proof would
immediately give a very fine picture of the subcritical regime.  In
particular, coupled with the results of this paper about the critical
regime, and the results of Bourgain-Goldstein about the supercritical
regime, a proof of the Almost Reducibility Conjecture would give a proof of
Conjecture \ref {sc}:
\begin{enumerate}
\item The Almost Reducibility Conjecture implies that the subcritical regime
can only support absolutely continuous spectrum \cite {A},
\item \cite {BG} implies that pure point spectrum is typical throughout the
supercritical regime,\footnote {More precisely, for every fixed potential,
and for almost every frequency, the spectrum is pure point with
exponentially decaying eigenfunctions throughout the region of the spectrum
where the Lyapunov exponent is positive.}
\item Theorem \ref {cod} implies that typically the critical regime is
invisible to the spectral measures.\footnote {Since in the continuation
of this series, \cite {A3}, we will show a stronger fact
(for fixed frequency, a typical potential has no critical energies), we just
sketch the argument.  For fixed frequency,
Theorem \ref {cod} implies that a typical potential admits at most countably
many critical energies.  Considering phase changes
$v_\theta(x)=v(x+\theta)$, which do not change the critical set, we see that
for almost every $\theta$ the critical set, being a fixed countable set,
can not carry any spectral weight (otherwise the average over $\theta$
of the spectral measures would have atoms, but this average has a
continuous distribution, the integrated density of states \cite {AS}).}
\end{enumerate}

We have so far, see \cite {A},
been able to prove this conjecture when $\alpha$ is
exponentially well approximated by rational numbers $p_n/q_n$:
$\limsup \frac {\ln
q_{n+1}} {q_n}>0$.  In the case of the almost Mathieu operator, the almost
reducibility conjecture was proved in \cite {AJ2}, \cite {A1} and \cite {A}.

\subsection{Further comments}

As mentioned before, it follows from the combination of \cite
{BJ1} and \cite {BJ2}, that the Lyapunov exponent is zero in the spectrum,
provided the potential is sufficiently small, irrespective of the frequency.
This is a very surprising result from the dynamical point of view.

For instance, fix some non-constant small $v$, and consider $\alpha$ close
to $0$.  Then the spectrum is close, in the Hausdorff topology, to the
interval
$[\inf v-2,\sup v+2]$.  However, if $E \notin [\inf v+2,\sup v-2]$
we have
\be
\lim_{n \to \infty} \lim_{\alpha \to 0} \frac {1} {n} \int_{\R/\Z} \ln
\|A^{(E-v)}(x+(n-1) \alpha) \cdots A(x)\| dx>0.
\ee
At first it might seem that as $\alpha \to 0$ the dynamics of
$(\alpha,A^{(E-v)})$ becomes
increasingly complicated and we should expect the
behavior of large potentials (with
positive Lyapunov exponents by \cite {SS}).\footnote
{In fact, the Lyapunov exponent function converges in the $L^1$-sense,
as $\alpha \to 0$, to a continuous function, positive outside $[\sup
v-2,\inf v+2]$ (see the argument of \cite {AD}).  This reconciles with the
fact that the edges of the spectrum (in two intervals
of size $\sup v-\inf v$) become increasingly thinner (in measure) as $\alpha
\to 0$.}  Somehow, delicate cancellation between expansion and contraction
takes place precisely at the spectrum and kills the Lyapunov exponent.

Bourgain-Jitomirskya's result
that the Lyapunov exponent must be zero on the spectrum in this situation
involves duality and localization arguments which are far from the dynamical
point of view.  Our work provides a different explanation for it, and
extends it from $\SL(2,\R)$-cocycles to $\SL(2,\C)$-cocycles.  Indeed,
quantization implies that all cocycles near constant have zero acceleration. 
Thus they are all regular.  Thus if $A$ is close to constant and
$(\alpha,A)$ has a positive Lyapunov exponent then it must
be uniformly hyperbolic.

We stress that while this argument explains why
constant cocycles are far from non-uniform hyperbolicity,
localization methods remain crucial to the
understanding of several aspects of the dynamics of
cocycles close to a constant one, at least in the Diophantine regime.

Let us finally make a few remarks and pose questions about the actual
values taken by the acceleration.
\begin{enumerate}
\item
If the coefficients of $A$ are trigonometric polynomials of degree at most
$n$, then $|\omega(\alpha,A)| \leq n$ by convexity
(since $L(\alpha,A_\epsilon) \leq \sup_{x
\in \R/\Z} \ln \|A(x+\epsilon i)\| \leq 2 \pi n \epsilon+O(1)$).
\item On the other hand, if $\alpha \in \R
\setminus \Q$, $|\lambda| \geq 1$ and $n \in \N$, then for $v(x)=2 \lambda
\cos 2 \pi n x$ we have $\omega(\alpha,A^{(E-v)})=n$ for every $E \in
\Sigma_{\alpha,v}$.  In the case $n=1$ (the almost Mathieu operator), this
is shown in the Appendix.  The general case reduces to this one since
for any $A \in C^\omega(\R/\Z,\SL(2,\C))$ and $n \in \N$,
$L(n \alpha,A(x))=L(\alpha,A(n x))$, which implies $n \omega(n
\alpha,A(x))=\omega(\alpha,A(n x))$.
\item If $\alpha \in \R \setminus \Q$ and
$A$ takes values in $\SO(2,\R)$, the acceleration is easily seen to be
the norm of the topological degree of $A$.
The results of \cite {AK2} imply
that this also holds for ``premonotonic cocycles'' which include
small $\SL(2,\R)$ perturbations of $\SO(2,\R)$ valued cocycles with non-zero
topological degree.
\item It seems plausible that the
norm of the topological degree is always a lower bound for the acceleration
of $\SL(2,\R)$ cocycles. 
In case of non-zero degree, is this bound achieved precisely by premonotonic
cocycles?
\item Consider a typical
perturbation of the potential
$2 \lambda \cos 2 \pi n x$, $\lambda>1$.  Do energies with any fixed
acceleration $1 \leq k \leq n$ form a set of positive measure?  It seems
promising to use the ``Benedicks-Carleson'' method of Lai-Sang Young \cite
{Y} to address aspects of this question ($k=n$, large $\lambda$, allowing
exclusion of small set of frequencies).  One is also tempted to relate the
acceleration to the number of ``critical points'' for the dynamics
(which can be identified when her method works).  Colisions between a few
critical points might provide a mechanism for the appearance of energies
with intermediate acceleration.
\end{enumerate}

\subsection{Outline of the remaining of the paper}

The outstanding issues (not covered in the introduction) are the proofs
of Theorems \ref {quantized}, \ref {uniformly hyperbolic} and \ref {cod1}.

We first address quantization (Theorem \ref {quantized}) in section \ref
{accel}.
The proof uses
periodic approximation.  A Fourier series estimate shows that as the
denominators of the approximations grow, quantization becomes more and more
pronounced.  The result then follows by
continuity of the Lyapunov exponent \cite {JKS}.

Next we show, in section \ref {reguh},
that regularity with positive Lyapunov exponent implies uniform
hyperbolicity (the hard part of Theorem \ref {uniformly hyperbolic}).  The proof again proceeds
by periodic approximation.  We first notice that the Fourier series estimate
implies that periodic approximants are uniformly hyperbolic, and hence have
unstable and stable directions.  If we can show that we can take an
analytic limit of those directions, then
the uniform hyperbolicity of $(\alpha,A)$ will follow.
A simple normality argument
shows that we only need to prove that the invariant directions do not get too
close as the denominators grow.  We show (by direct computation) that if
they would get too close, then the derivative of the Lyapunov exponent would
be relatively large with respect to perturbations of some Fourier modes of
the potential.  This contradicts a ``macroscopic'' bound on the derivative
which comes from pluriharmonicity.

We then show, in section \ref {loc},
the non-vanishing of the
derivative of the canonical analytic extension of the
Lyapunov exponent, $L_{\delta,j}$ (Theorem \ref {cod1}).
Under the hypothesis that $\omega(\alpha,A^{(v_*)})=j>0$,
$(\alpha,A^{(v_*)}_{\delta'})
\in \UH$ for $0<\delta'<\delta_0$ ($0<\delta_0<\delta$ small),
so we can define holomorphic invariant directions $u$ and $s$,
over $0<\Im z<\delta_0$.  Using the explicit expressions
for the derivative of the Lyapunov exponent
in terms of the unstable and stable directions $u$ and $s$,
derived in section \ref {reguh},
we conclude that the vanishing of the derivative would imply a symmetry of
Fourier coefficients (of a suitable expression involving $u$ and $s$), which
is enough to conclude that $u$ and $s$
analytically continuate through $\Im
z=0$.  This implies that $(\alpha,A^{(v_*)})$ is ``conjugate to a cocycle of
rotations'', which implies that its acceleration is zero, contradicting the
hypothesis.

We also include two appendices.  The first gives the basic facts about
uniformly hyperbolic cocycles, especially regarding the regularity of the
Lyapunov exponent.  The second shows

We also include an appendix showing how to use quantization to compute the
Lyapunov exponent and acceleration in the case of the almost
Mathieu operator, which is used in deriving the Example Theorem.

{\bf Acknowledgements:} I am grateful to Svetlana Jitomirskaya and David
Damanik for several detailed comments which greatly improved the exposition.

\section{Quantization of acceleration: proof of Theorem \ref {quantized}}
\label {accel}

We will use the continuity in the frequency of the Lyapunov exponent \cite
{BJ1}, \cite {JKS}.\footnote {
Bourgain-Jitomirskaya actually restricted considerations to the case of
Schr\"odinger (in particular, $\SL(2,\R)$ valued) cocycles.  Their result
was generalized to the $\SL(2,\C)$ case in the work of
Jitomirskaya-Koslover-Schulteis \cite {JKS}.}

\begin{thm}[\cite {JKS}] \label {con}

If $A \in C^\omega(\R/\Z,\SL(2,\C))$, then
the $\alpha \mapsto L(\alpha,A)$, $\alpha \in \R$, is continuous at every
$\alpha \in \R \setminus \Q$.

\end{thm}

This result is very delicate:
the restriction of $\alpha \mapsto L(\alpha,A)$ to $\R \setminus \Q$ is not,
in general, uniformly continuous.

Notice that if $p/q$ is a rational number, then
there exists a simple expression for the Lyapunov exponent $L(p/q,A)$
\be
L(p/q,A)=\frac {1} {q} \int_{\R/\Z} \ln \rho(A_{(p/q)}(x)) dx
\ee
where $A_{(p/q)}(x)=A(x+(q-1) p/q) \cdots A(x)$ and $\rho(B)$ is the
spectral radius of an $\SL(2,\C)$ matrix
$\rho(B)=\lim_{n \to \infty} \|B^n\|^{1/n}$.  A key observation is that
if $p$ and $q$ are coprime then
the trace $\tr A_{(p/q)}(x)$ is a $1/q$-periodic function of
$x$.  This follows from the relation
\be
A(x) A_{(p/q)}(x)=A_{(p/q)}(x+p/q) A(x),
\ee
expressing the fact that $A_{(p/q)}(x)$ and $A_{(p/q)}(x+p/q)$ are conjugate
in $\SL(2,\C)$, and hence $A_{(p/q)}(x)$ is conjugate to $A_{(p/q)}(x+k
p/q)$ for any $k \in \Z$.

Fix $\alpha \in \R \setminus \Q$ and $A \in C^\omega(\R/\Z,\SL(2,\C))$
and let $p_n/q_n$ be a sequence of rational
numbers ($p_n$ and $q_n$ coprime) approaching $\alpha$ (not necessarily
continued fraction approximants).

Let $\epsilon>0$ and $C>0$ be such that
$A$ admits a bounded extension to $|\Im z|<\epsilon$ with $\sup_{|\Im
z|<\epsilon} \|A(z)\|<C$.
Since $\tr A_{(p_n/q_n)}$ is $1/q_n$-periodic,
\be
\tr A_{(p_n/q_n)}(x)=\sum_{k \in \Z} a_{k,n} e^{2 \pi i k q_n x},
\ee
with $a_{k,n} \leq 2 C^{q_n} e^{-2 \pi k q_n \epsilon}$.

Fix $0<\epsilon'<\epsilon$.  Fixing $k_0$ sufficiently large, we get
\be
\tr A_{(p_n/q_n)}(x)=\sum_{|k| \leq k_0}
a_{k,n} e^{2 \pi i k q_n x}+O(e^{-q_n}), \quad
|\Im x|<\epsilon',
\ee
for $n$ large.  Since $\max \{0,\frac {1} {2} \tr\} \leq \ln \rho \leq \max
\{0,\tr\}$, it follows that
\be
L(p/q,A_\delta)=\max_{k \leq |k_0|}
\max \{\ln |a_{k,n}|-2 \pi k \delta,0\}+o(1),
\quad \delta<\epsilon'.
\ee
Thus for large $n$, $\delta \mapsto L(p_n/q_n,A_\delta)$ is close, over
$|\delta|<\epsilon'$, to a convex piecewise linear function with
slopes in $\{-2 \pi k_0,...,2 \pi k_0\}$.
By Theorem \ref {con}, these functions
converge uniformly on compacts of
$|\delta|<\epsilon$ to $\delta \mapsto L(\alpha,A_\delta)$.  It follows
that $\delta \mapsto L(\alpha,A_\delta)$ is
a convex piecewise linear function of $|\delta|<\epsilon'$,
with slopes in $\{-2 \pi k_0,...,2 \pi k_0\}$, so
$\omega(\alpha,A) \in \Z$.\qed

\begin{rem} \label {rational}

Consider say $A(x)=\left (\bm e^{\lambda(x)} & 0 \\ 0 & e^{-\lambda(x)} \em \right )$
with $\lambda(x)=e^{2 \pi i q_0 x}$ for some $q_0>0$.  Then
$L(\alpha,A_\epsilon)=\frac {2} {\pi} e^{-2 \pi q_0 \epsilon}$ if
$\alpha=p/q$ for some $q$ dividing $q_0$, and $L(\alpha,A_\epsilon)=0$
otherwise.  This gives an example both of discontinuity of the Lyapunov
exponent and of lack of quantization of acceleration at rationals.

If we had chosen $\lambda$ as a more typical function of zero average, we
would get discontinuity of the Lyapunov exponent and lack of quantization at
all rationals, both becoming increasingly less pronounced as the
denominators grow.

\end{rem}

\section{Characterization of uniform hyperbolicity: proof of Theorem \ref
{uniformly hyperbolic}} \label {reguh}

Since the Lyapunov exponent is a $C^\infty$ function in $\UH$, the ``if''
part is obvious from quantization.
In order to prove the ``only if'' direction, we will
first show the uniform hyperbolicity of periodic
approximants and then show that uniform hyperbolicity persists in the limit.
To do this last part, we will use an explicit formula for
the derivative of the Lyapunov exponent (fixed frequency) in $\UH$.

\subsection{Uniform hyperbolicity of approximants}

\begin{lemma} \label {per}

Let $(\alpha,A) \in (\R \setminus \Q) \times C^\omega(\R/\Z,\SL(2,\R))$
and assume that $(\alpha,A_\delta)$ is regular with positive Lyapunov
exponent.  If $p/q$ is close to
$\alpha$ and $\tilde A$ is close to $A$ then $(p/q,\tilde A)$ is uniformly
hyperbolic.

\end{lemma}

\begin{proof}

Let us show that if $p_n/q_n \to \alpha$ and $A^{(n)} \to A$ then there exists
$\epsilon''>0$ such that
\be \label {201}
\frac {1} {q_n} \ln \rho (A^{(n)}_{(p_n/q_n)}(x))=L(\alpha,A_{\Im x})+o(1),
\quad |\Im x|<\epsilon'',
\ee
which implies the result.  In fact this estimate is just a slight adaptation
of what we did in section \ref {accel}.

Since $A^{(n)} \to A$ and $A$ is regular, we may
choose $\epsilon>0$ such that
$(\alpha,A_\delta)$ is regular for $|\delta|<\epsilon$, $A_n \in
C^\omega_\epsilon(\R/\Z,\SL(2,\C))$ for every $n$ and $A_n \to A$ uniformly
in $|\Im z|<\epsilon$.

Choose $\epsilon''<\epsilon'<\epsilon$.
We have seen in section \ref {quantized} that there exists $k_0$ such that
\be \label {202}
\tr A^{(n)}_{(p_n/q_n)}(x)=\sum_{|k| \leq k_0}
a_{k,n} e^{2 \pi i k q_n x}+O(e^{-q_n}), \quad
|\Im x|<\epsilon',
\ee
\be \label {203}
L(p_n/q_n,A^{(n)}_\delta)=\max_{k \leq |k_0|}
\max \{\ln |a_{k,n}|-2 \pi k \delta,0\}+o(1),
\quad |\delta|<\epsilon'.
\ee

By Theorem \ref {con},
$L(p_n/q_n,A^{(n)}_\delta) \to L(\alpha,A_\delta)$ uniformly on compacts of
$|\delta|<\epsilon$, so we may rewrite (\ref {203}) as
\be \label {204}
L(\alpha,A^{(n)}_\delta)=\max_{k \leq |k_0|}
\max \{\ln |a_{k,n}|-2 \pi k \delta,0\}+o(1),
\quad |\delta|<\epsilon'.
\ee
Since the left hand side in (\ref {204}) is an affine positive function of
$\delta$, with slope $2 \pi \omega(\alpha,A)$, over $|\delta|<\epsilon$,
it follows that $|\omega(\alpha,A)| \leq k_0$,
\be \label {205}
L(\alpha,A_\delta)=\ln |a_{-\omega(\alpha,A),n}|+2 \pi \omega(\alpha,A)
\delta+o(1), \quad |\delta|<\epsilon'',
\ee
and morever, if $|j| \leq k_0$ is such that $j \neq -\omega(\alpha,A)$
we have
\be \label {207}
\ln |a_{j,n}|-2 \pi j \delta+2 \pi
(\epsilon''-\epsilon') \leq L(\alpha,A_\delta)+o(1),
\quad |\delta|<\epsilon''.
\ee
Together, (\ref {202}), (\ref {205}) and (\ref {207}) imply (\ref {201}),
as desired.
\end{proof}

\subsection{Derivative of the Lyapunov exponent at uniformly hyperbolic
cocycles}

Fix $(\alpha,A) \in \UH$.  Let $u,s:\R/\Z \to \P \C^2$ be the unstable and
stable directions.

Let $B:\R/\Z \to \SL(2,\C)$ be analytic with column vectors in the
directions of $u(x)$ and $s(x)$.  Then
\be
B(x+\alpha)^{-1}A(x)B(x)=\left (\bm \lambda(x)&0\\0&\lambda(x)^{-1} \em
\right )=D(x).
\ee
Obviously $L(\alpha,A)=L(\alpha,D)$, and $\int \Re \ln
\lambda(x) dx=L(\alpha,A)$.\footnote {
Notice that the quantization of the acceleration, in the uniformly
hyperbolic case, follows immediately from this expression
(the integer arising being the number of turns $\lambda(x)$ does
around $0$).}

Write $B(x)=\left (\bm a(x)&b(x)\\c(x)&d(x) \em \right )$.  We note that
though the definition of $B$ involves arbitrary choices, it is clear that
$q_1(x)=a(x) d(x)+b(x) c(x)$,
$q_2(x)=c(x) d(x)$ and $q_3(x)=-b(x) a(x)$ depend only on $(\alpha,A)$.
We will call
$q_i$, $i=1,2,3$, the {\it coefficients of the derivative of the Lyapunov
exponent}, for reasons that will be clear in a moment.

\begin{lemma}

Let $(\alpha,A) \in \UH$ and let $q_1,q_2,q_3:\R/\Z \to \C$ be the
coefficients of the derivative of the Lyapunov exponent.
Let $w:\R/\Z \to \sl(2,\C)$ be analytic, and write $w=\left (\bm w_1 & w_2
\\ w_3 & -w_1 \em \right )$.  Then
\be
\frac {d} {dt} L(\alpha,A e^{t w})=\Re \int_{\R/\Z} \sum_{i=1}^3 q_i(x) w_i(x)
dx, \quad \text {at} \quad t=0.
\ee

\end{lemma}

\begin{proof}

Write $B(x+p/q)^{-1} A(x) e^{t w(x)} B(x)=D^t(x)$.  We notice that
\be \label {d1}
D(x)^{-1} \frac {d} {dt} D^t(x)=B(x)^{-1} w(x) B(x), \quad \text {at} \quad
t=0,
\ee
and
\be \label {d2}
\sum_{i=1}^3 q_i(x) w_i(x)=\text {u.l.c. of } B(x)^{-1} w(x) B(x),
\ee
where u.l.c. stands for the upper left coefficient.

Suppose first that $\alpha$ is a rational number $p/q$.  Then
\be
\frac {d} {dt} L(p/q,A e^{t w})=\frac {1} {q} \frac {d} {dt} \int_{\R/\Z}
\ln \rho(D^t_{(p/q)}(x)) dx,
\ee
so it is enough to show that
\be \label {d3}
\frac {d} {dt}
\ln \rho(D^t_{(p/q)}(x))=\Re \sum_{j=0}^{q-1}
\sum_{i=1}^3 q_i(x+j p/q) w_i(x+j p/q), \quad \text {at} \quad t=0.
\ee
Since $D_{(p/q)}(x)$ is diagonal and its u.l.c.
has norm bigger than $1$,
\be \label {d4}
\frac {d} {dt} \ln \rho(D^t_{(p/q)}(x))=\Re \text { u.l.c. of }
D_{(p/q)}(x)^{-1} \frac {d} {dt} D^t_{(p/q)}(x), \quad \text {at}
\quad t=0.
\ee
Writing $D_{[j]}(x)=D(x+(j-1) p/q)
\cdots D(x)$, and using (\ref {d1}), we see that
\begin{align} \label {d5}
&D_{(p/q)}(x)^{-1} \frac {d} {dt}
D^t_{(p/q)}(x)\\
\nonumber
&=\sum_{j=0}^{q-1} D_{[j]}(x)^{-1} B(x+j p/q)^{-1} w(x+j p/q) B(x+j p/q)
D_{[j]}(x),
\quad \text {at} \quad t=0.
\end{align}
Since the $D_{[j]}$ are diagonal,
\begin{align} \label {d7}
\text {u.l.c. of }
D_{[j]}(x)^{-1} &B(x+j p/q)^{-1} w(x+j p/q) B(x+j p/q)
D_{[j]}(x)\\
\nonumber
&=\text {u.l.c. of } B(x+j p/q)^{-1} w(x+j p/q) B(x+j p/q).
\end{align}
Putting together (\ref {d2}), (\ref {d4}), (\ref {d5})
and (\ref {d7}), we get (\ref {d3}).

The validity of the formula in the rational case yields the irrational
case by approximation (since the Lyapunov exponent is $C^\infty$ in $\UH$).
\end{proof}

\subsection{Proof of Theorem \ref {uniformly hyperbolic}}

Let $(\alpha,A) \in (\R \setminus \Q) \times
C^\omega(\R/\Z,\SL(2,\C))$ be such that $(\alpha,A)$ is regular.  Then there
exists $\epsilon>0$ such that $L(\alpha,A_\delta)$ is regular for
$|\delta|<\epsilon$.

Fix $0<\epsilon'<\epsilon$.
Choose a sequence $p_n/q_n \to \alpha$.
By Lemma \ref {per}, if $n$ is large then
$(p_n/q_n,A_\delta)$ is uniformly hyperbolic for $\delta<\epsilon'$.
So one can define functions
$u_n(x),s_n(x)$ with values in $\P \C^2$, corresponding to the
eigendirections of $A_{(p_n/q_n)}(x)$ with the largest and smallest
eigenvalues.  Our strategy will be to show that the sequences
$u_n(x)$ and $s_n(x)$ converge uniformly (in a band)
to functions $u(x)$ and $s(x)$.

The coefficients of the derivative of $L(p_n/q_n,A)$ will be denoted
$q^n_i$, $i=1,2,3$.  The basic idea now is that if $q^n_2(x)$ and $q^n_3(x)$
are bounded, then it follows directly from the definitions that
the angle between $u_n(x)$ and $s_n(x)$ is not too small, and this is enough
to guarantee convergence.  On the other
hand, the derivative of the Lyapunov exponent is under control by
pluriharmonicity, which yields the desired bound on the coefficients.

There are various way to proceed here, and we will just do an
estimate of the Fourier coefficients of the $q^n_j$, $j=2,3$.
Write
\be
\zeta_{n,k,j}=\int_{\R/\Z} q^n_j(x) e^{2 \pi i k x} dx.
\ee

\begin{lemma} \label {gam}

There exist $C>0$,
$\gamma>0$ such that for every $n$ sufficiently large,
\be
|\zeta_{n,k,j}| \leq C e^{-\gamma |k|}, \quad j=2,3, \quad k \in \Z.
\ee

\end{lemma}

\begin{proof}

Choose $0<\gamma<2 \pi \epsilon'$.  Then for each fixed
$n$ large we have
$|\zeta_{n,k,j} \leq C_n e^{-\gamma |k|}$ (since $q^n_j$ extend to $|\Im
z|<\epsilon'$).
If the result did not hold, then
there would exist $n_l \to \infty$, $k_l \in \Z$, $j_l=2,3$ such that
$|\zeta_{n_l,k_l,j_l}|>l e^{-\gamma |k_l|}$.
We may assume that $j_l$ is a constant and
either $k_l>0$ for all $l$ or $k_l \leq 0$ for all $l$.

For simplicity, we will assume that $j_l=2$ and
$k_l \leq 0$ for all $l$.  Let
\be
w_{(l)}(x)=\frac {|\zeta_{n_l,k_l,2}|}
{\zeta_{n_l,k_l,2}} e^{\gamma |k_l|}
\left (\bm 0&e^{2 \pi i k_l x}
\\ 0 & 0 \em \right ).
\ee
Choose $\gamma<\gamma'<2 \pi \epsilon'$.
Setting $\tilde A(x)=A(x+i \gamma'/2 \pi)$ and
$\tilde w_{(l)}(x)=w_{(l)}(x+i \gamma'/2\pi)$, we get
\be \label {l}
\frac {d} {dt} L(p_{n_l}/q_{n_l},\tilde A e^{t \tilde w_{(l)}})=
e^{\gamma |k_l|} |\zeta_{n_l,k_l,2}| \geq l,
\ee
since the coefficients of the derivative at $(p_{n_l}/q_{n_l},\tilde A)$ are
$\tilde q^{n_l}_j(x)=q^{n_l}_j(x+i\gamma'/2 \pi)$.

Notice that $\tilde w_{(l)}$ admits a holomorphic extension bounded by
$1$ on $|\Im z|<(\gamma'-\gamma)/2 \pi$.  Since $(\alpha,\tilde A)$ is
regular with positive Lyapunov exponent, it follows from Lemma \ref {per},
that the exists $r>0$ such that for every $l$ large $(p_{n_l}/q_{n_l},\tilde
A e^{t \tilde w_{(l)}})$ is uniformly hyperbolic for complex $t$ with
$|t|<r$.  In particular, the functions
$t \mapsto L(p_{n_l}/q_{n_l},\tilde A e^{t \tilde w_{(l)}})$ are harmonic on
$|t|<r$ for $l$ large.  Those functions are also clearly uniformly bounded. 
Harmonicity gives then that the derivative at $t=0$ is uniformly bounded as
well.  This contradicts (\ref {l}).
\end{proof}

\begin{lemma} \label {ang}

If $a,b,c,d \in \C$ are such that $ad-bc=1$, and the angle between the
complex lines through $\left (\bm a\\c \em \right )$ and
$\left (\bm b \\ d \em \right )$ is small, then $\max \{|ab|,|cd|\}$ is
large.

\end{lemma}

\begin{proof}

Straightforward computation.
\end{proof}

Lemma \ref {gam} implies that there exists $\gamma>0$ such that $q^n_2$ and
$q^n_3$ are uniformly bounded, as $n \to \infty$, on $|\Im x|<\gamma$.  By
Lemma \ref {ang}, this implies that there exists $\eta>0$ such that the angle
between $u_n(x)$ and $s_n(x)$ is at least $\eta$, for every $n$ large and
$|\Im x|<\gamma$.  We are in position to apply a
normality argument.

\begin{lemma}

Let $u_n(x)$ and $s_n(x)$ be holomorphic functions defined in some complex
manifold, with values in the
$\P\C^2$.  If the angle between $u_n(x)$ and $s_n(x)$ is bounded away from
$0$ for every $x$ and $n$, then $u_n(x)$ and $s_n(x)$ form normal families,
and limits of $u_n$ and of $s_n$ (taken along the same subsequence) are
holomorphic functions such that $u(x) \neq s(x)$ for every $x$.

\end{lemma}

\begin{proof}

We may identify $\P\C^2$ with the Riemann Sphere.
Write $\phi_n(x)=u_n(x)/s_n(x)$.  Then $\phi_n(x)$ avoids a neighborhood of
$1$, hence it forms a normal family.  Let us now take a sequence along which
$\phi_n$ converges, and let us show $u_n(x)$ and $s_n(x)$ form normal
families.  This is a local problem, so we may work in a neighborhood of a
point $z$.  If $\lim \phi_n(z) \neq \infty$, then
for every $n$ large $\phi_n$ must be bounded (uniformly in a neighborhood of
$z$), so $u_n$ and $1/s_n$ must also be bounded.
If $\lim \phi_n(z)=\infty$, then
for every $n$ large $1/\phi_n$ must be bounded (uniformly in a neighborhood
of $z$), so $s_n$ and $1/u_n$ must be bounded.  In either case we conclude
that $s_n$ and $u_n$ are normal in a neighborhood of $z$.

The last statement is obvious by pointwise convergence.
\end{proof}

Let $u(x)$ and $s(x)$ be limits of $u_n(x)$ and $s_n(x)$ over $|\Im
x|<\gamma$, taken along the same subsequence.  Then $A(x) \cdot
u(x)=u(x+\alpha)$, $A(x) \cdot s(x)=s(x+\alpha)$ and $u(x) \neq s(x)$. 
Since $\alpha \in \R \setminus \Q$ and
$L(\alpha,A)>0$, this easily implies that $(\alpha,A) \in \UH$.
\qed

\section{Local non-triviality of the Lyapunov function in strata: Proof of
Theorem \ref {cod1}} \label {loc}

Let $\delta,j,v_*$ be as in the statement of Theorem \ref {cod1}.
Notice that $(\alpha,A^{(v_*)}) \notin \UH$, since otherwise we would have
$j=\omega(\alpha,A^{(v_*)})=0$.

Let $0<\epsilon_0<\delta$ be such that
$(\alpha,A^{(v_*)}_\epsilon) \in \UH$ and
$\omega(\alpha,A^{(v_*)}_\epsilon)=j$ for $0<\epsilon<\epsilon_0$.
By definition, for every $0<\epsilon<\epsilon_0$, we have
$L_{\delta,j}(\alpha,A)=L(\alpha,A_\epsilon)-2 \pi j \epsilon$ for $v$ in a
neighborhood of $v_*$.

Let $u,s:\{0<\Im x<\epsilon_0\} \to \P\C^2$ be
such that $x \mapsto u(x+i \epsilon)$ and $x \mapsto
s(x+i \epsilon)$ are the unstable and
stable directions of $(\alpha,A^{(v_*)}_\epsilon)$, and let
$q_1,q_2,q_3:\{0<\Im x<\epsilon_0\} \to \C$
be such that $x \mapsto q_j(x+i \epsilon)$ is the
$j$-th coefficient of the derivative of $(\alpha,A^{(v_*)}_\epsilon)$.
Due to the Schr\"odinger form, it is immediate to check that
$q_2(x)=-q_3(x-\alpha)$.

Notice that $A^{(v_*+w)}=A^{(v_*)} e^{\tilde w}$, where $\tilde w(x)=\left
(\bm 0 & 0 \\ -w(x) & 0 \em \right )$.
Thus the derivative of $w \mapsto
L_{\delta,j}(\alpha,A^{(v_*+t w)})$ at $t=0$ is
\be \label {q_3}
\Re \int_{\R/\Z} -w(x+i \epsilon) q_3(x+i \epsilon) dx.
\ee

If the result does not hold,
then (\ref {q_3}) must vanish for every $w \in C^\omega_\delta(\R/\Z,\R)$.
Testing this with $w$ of the form $a \cos 2 \pi k x+b \sin 2 \pi k x$, $a,b
\in \R$, $k \in \Z$, we see that the $k$-th Fourier coefficient of $q_3$ must
be minus the complex conjugate of the $-k$-th Fourier coefficient of $q_3$ for
every $k \in \Z$.  Since the Fourier series converges for $0<\Im
x<\epsilon_0$, this implies that it actually converges for $|\Im
x|<\epsilon_0$, and at $\R/\Z$ it defines a purely imaginary function.
Thus $q_3(x)$ extends analytically through $|\Im x|=0$, and hence
$q_2(x)=c(x)d(x)=a(x-\alpha)b(x-\alpha)=-q_3(x-\alpha)$ (the middle equality
holding due to the Schr\"odinger form) also does.

Identifying $\P \C^2$ with the Riemann sphere in the usual way (the line
through $\left (\bm z \\ w \em \right )$ corresponding to $z/w$), we get
$q_2=\frac {1} {u-s}$ and $q_3=\frac {us} {u-s}$.  These formulas allow us
to analytically continuate $u$ and $s$ through $\Im x=0$.  Since $q_2$ and
$q_3$ are purely imaginary at $\Im x=0$, we conclude that $u$ and $s$ are
complex conjugate directions in $\P \C^2$, and since they are distinct they
are also non-real.

Let $B(x) \in \SL(2,\R)$ be the unique upper triangular matrix taking the
pair $u(x)$ and $s(x)$ to $\left (\bm \pm i \\ 1 \em \right )$.  Then
$B:\R/\Z \to \SL(2,\R)$ is analytic.  Define
$A(x)=B(x+\alpha) A^{(v_*)}(x) B(x)^{-1}$.  Since $A^{(v_*)}(x)$ takes
$u(x)$ and $s(x)$ to $u(x+\alpha)$ and $s(x+\alpha)$, we conclude that
$B(x+\alpha) A(x) B(x)^{-1} \in \SO(2,\R)$.
Since $x \mapsto A^{(v_*)}(x)$
is homotopic to a constant as a function $\R/\Z \to
\SL(2,\R)$, $x \mapsto B(x+\alpha) A^{(v_*)}(x) B(x)^{-1}
\in \SL(2,\R)$ is homotopic to a constant as a function
$\R/\Z \to \SL(2,\R)$, thus also as a function $\R/\Z \to \SO(2,\R)$.
It follows that there exists an analytic function $\phi:\R/\Z \to \R$ such that
$B(x+\alpha) A(x) B(x)^{-1}=A(x)$, where $A(x)$
is the rotation of angle $2 \pi \phi(x)$.

Obviously this relation implies that
$L(\alpha,A^{(v_*)}_\epsilon=L(\alpha,A_\epsilon)$ for $\epsilon>0$ small. 
If we show that $L(\alpha,A_\epsilon)=0$ for $\epsilon>0$ small,
we will conclude that $\omega(\alpha,A^{(v_*)})=0$, contradicting
the hypothesis.

To see that $L(\alpha,A_\epsilon)=0$, notice that for $n \geq 1$,
$A_n(x)$ is the rotation
of angle $\sum_{k=0}^{n-1} \phi(x+k \alpha)$.  Thus
\be
\frac {1} {n} \int_{\R/\Z}
\ln \|A_n(x+i \epsilon)\| dx=2 \pi \int_{\R/\Z} \frac {1} {n}
|\sum_{k=0}^{n-1} \Im \phi(x+k \alpha+i \epsilon)| dx.
\ee
Since $x \mapsto x+\alpha$ is ergodic with respect to Lebesgue measure, the
integrand of the right hand side converges uniformly,
as $n \to \infty$, to $|\int_{\R/\Z} \Im \phi(x+i \epsilon)
dx|=|\int_{\R/\Z} \Im \phi(x) dx|=0$.  Thus the limit of the right hand
side, which is $L(\alpha,A_\epsilon)$ by definition, is zero as well.\qed

\begin{rem} \label {ak2}

The analysis of the function $A \mapsto L_{\delta,j}(\alpha,A)$ on
$C^\omega_\delta(\R/\Z,\SL(2,\R)) \to \R$, with $\alpha \in \R \setminus
\Q$ and near some $A^*$ with $\omega(\alpha,A^*)>0$ can be carried out as
above with one important difference.

The argument above does allow one to establish that if $L_{\delta,j}$ is not
a local submersion, then the coefficients of the derivative $q_2$ and $q_3$
extend from some half band $0<\Im x<\epsilon_0$ to a full band $|\Im
x|<\epsilon_0$.\footnote {Though one lacks the symmetry between $q_2$ and $q_3$
exploited above, we just separately evaluate the extensions of
$q_2$ and $q_3$, since we are not constrained to consider just perturbations
of a specific form.}  This again leads to the conclusion that there exists
$B:\R/\Z \to \SL(2,\R)$ analytic such that
$A(x)=B(x+\alpha) A^*(x) B(x)^{-1}$ takes values in $\SO(2,\R)$.  But now
there are two cases.
\begin{enumerate}
\item $x \mapsto A^*(x)$ is homotopic to a constant.  In this case, the
above argument goes through and one concludes that $\omega(\alpha,A^*)=0$,
contradiction.
\item $x \mapsto A^*(x)$ is not homotopic to a constant.  In this case,
there is no contradiction, and the reader is invited to check that if
$A^*(x)$ is the rotation of angle $2 \pi x$ then indeed the derivative of
$L_{\delta,j}$ vanishes, though $\omega(\alpha,A^*)=1$.
\end{enumerate}

The analysis of the second case has been carried out
by different means in \cite {AK2}, where it is shown that the Lyapunov
exponent is real analytic near cocycles with values in $\SO(2,\R)$ provided
they are not homotopic to a constant.  We should emphasize that
this result is obtained for any number of frequencies, which is certainly
beyond the scope of the techniques we develop in this paper.

Interpreting their results (in the one-frequency case)
from our new point of view, \cite {AK2} shows
that all real perturbations of $(\alpha,A^*)$ have the same acceleration
(the absolute value of the
topological degree of $A^*$ as a map $\R/\Z \to \SL(2,\R)$).  Real
analyticity implies that the derivative
of the Lyapunov exponent then is forced to vanish whenever the Lyapunov
exponent is zero.  In \cite {AK2} it is shown that the second derivative is
non-zero.  The locus of zero exponents can be shown to intersect a
neighborhood of $A^*$ in $C^\omega_\delta(\R/\Z,\SL(2,\R))$
in an analytic submanifold of
codimension $4 |\omega(\alpha,A^*)|$.

Thus our result implies that among cocycles non-homotopic to constants (and
with a given irrational frequency), the locus of
zero exponents is contained in a countable union of positive codimension
submanifolds of $C^\omega_\delta(\R/\Z,\SL(2,\R))$.

\end{rem}

\appendix

\section{Some almost Mathieu computations}

Through this section, we let $v(x)=2 \cos 2 \pi x$.

\begin{thm} \label {am1}

If $\alpha \in \R \setminus \Q$, $\lambda>0$, $E \in \R$ and $\epsilon \geq
0$ then $L(\alpha,A^{(E-\lambda v)}_\epsilon)=\max
\{L(\alpha,A^{(E-\lambda v)}),(\ln \lambda)+2 \pi \epsilon\}$, $\epsilon \geq 0$.

\end{thm}

\begin{proof}

A direct computation shows that
if $E$ and $\lambda$ are fixed then for every $\delta>0$, there exists
$0<\xi<\pi/2$ such that if $\epsilon$ is large and $w \in \C^2$
makes angle at most $\xi$ with the horizontal line then for every $x \in
\R/\Z$,
$w'=A^{(E-\lambda v)}(x+\epsilon i) \cdot w$ makes angle at most $\xi/2$ with
the horizontal line and $|\ln \|w'\|-(\ln \lambda+2 \pi
\epsilon)|<\delta$.

This implies that $L(\alpha,A^{(E-\lambda
v)}_\epsilon)=2 \pi \epsilon+\ln \lambda+o(1)$ as $\epsilon \to \infty$.
By quantization of acceleration, for every $\epsilon$ sufficiently large,
$\omega(\alpha,A^{(E-\lambda v)}_\epsilon)=1$ and
$L(\alpha,A^{(E-\lambda v)}_\epsilon)=2 \pi \epsilon+\ln \lambda$.
By real-symmetry,
$\omega(\alpha,A^{(E-\lambda v)}_\epsilon)$ is either $0$ or $1$ for
$\epsilon \geq 0$.  This implies the given formula for
$L(\alpha,A^{(E-\lambda v)}_\epsilon)$.
\end{proof}

For completeness, let us give a contrived rederivation of
the Aubry-Andr\'e formula.

\begin{cor}[\cite {BJ1}] \label {am2}

If $\alpha \in \R \setminus \Q$, $\lambda>0$, $E \in \R$
then $L(\alpha,A^{(E-\lambda v)}) \geq \max \{0,\ln \lambda\}$
with equality if and only if
$E \in \Sigma_{\alpha,v}$.

\end{cor}

\begin{proof}

The complement of the spectrum consists precisely of energies with
positive Lyapunov exponent and zero acceleration (as those two properties
characterize uniform hyperbolicity for $\SL(2,\R)$-valued cocycles by
Theorem \ref {uniformly hyperbolic}).

The previous theorem gives the inequality, and shows that it is strict if
and only if $L(\alpha,A^{(E-\lambda v)})>0$ and $\omega(\alpha,A^{(E-\lambda
v)})=0$.
\end{proof}

\subsection{Proof of the Example Theorem}

Fix $\alpha \in \R \setminus \Q$, $\lambda>1$ and $w \in
C^\omega_\delta(\R/\Z,\R)$.  Let
$v_\epsilon=\lambda v+\epsilon w$.

\begin{lemma}

If $\epsilon$ is sufficiently small, and
$E \in \Sigma_{\alpha,v_\epsilon}$ then
$\omega(\alpha,A^{(E-v_\epsilon)})=1$.

\end{lemma}

\begin{proof}

By Theorem \ref {am1} and Corollary \ref {am2},
$L(\alpha,A^{(E-\lambda v)}) \geq \ln \lambda$ and
$\omega(\alpha,A^{(E-\lambda v)}) \leq 1$ for every $E \in \R$.

For $\epsilon$ small we have $\Sigma_{\alpha,v_\epsilon} \subset
[-4\lambda,4\lambda]$.  By continuity of the Lyapunov exponent and
upper semicontinuity of the acceleration, we get
$\omega(\alpha,A^{(E-v_\epsilon)}) \leq 1$ and
$L(\alpha,A^{(E-v_\epsilon)})>0$ for every $E \in
\Sigma_{\alpha,v_\epsilon}$.

Since $A^{(E-v_\epsilon)}$ is real symmetric,
$\omega(\alpha,A^{(E-v_\epsilon)}) \geq 0$ as well, and if
$\omega(\alpha,A^{(E-v_\epsilon)})=0$ with $E \in
\Sigma_{\alpha,v_\epsilon}$ then $(\alpha,A^{(E-v_\epsilon)})$ is regular. 
This last possibility can not happen: since the Lyapunov exponent is
positive, it would imply uniform hyperbolicity, which can not happen in the
spectrum.  We conclude that $\omega(\alpha,A^{(E-v_\epsilon)})>0$ for $E \in
\Sigma_{\alpha,v_\epsilon}$.  By quantization, this forces
$\omega(\alpha,A^{(E-v_\epsilon)})=1$.
\end{proof}

By Proposition \ref {pluri}, $E \mapsto
L(\alpha,A^{(E-v_\epsilon)})$ coincides in the spectrum with the restriction
of an analytic function ($E \mapsto
L_{\delta,1}(\alpha,A^{(E-v_\epsilon)})$) defined in some neighborhood. 
This concludes the proof of the Example Theorem.\qed

\end{document}